\documentclass[11pt]{amsart}

\usepackage{amsmath, amsthm, amsfonts, amssymb, xy,  mathrsfs, graphicx, lscape, array, booktabs}
\usepackage{enumerate}
\usepackage[usenames, dvipsnames]{xcolor}
\usepackage[margin=1in]{geometry} 
\usepackage{tikz}
\usetikzlibrary{matrix,arrows}
\numberwithin{equation}{section}
\newtheorem{theorem}[equation]{Theorem}

\newtheorem{proposition}[equation]{Proposition}
\newtheorem{lemma}[equation]{Lemma}
\newtheorem{corollary}[equation]{Corollary}
\newtheorem{conjecture}[equation]{Conjecture}
\newtheorem{problem}[equation]{Problem}
\newtheorem{notation}[equation]{Notation}

\theoremstyle{definition}
\newtheorem{rmk}[equation]{Remark}
\newenvironment{remark}[1][]{\begin{rmk}[#1] \pushQED{\qed}}{\popQED \end{rmk}}
\newtheorem{eg}[equation]{Example}
\newenvironment{example}[1][]{\begin{eg}[#1] \pushQED{\qed}}{\popQED \end{eg}}
\newtheorem{defn}[equation]{Definition}
\newenvironment{definition}[1][]{\begin{defn}[#1]\pushQED{\qed}}{\popQED \end{defn}}
\newtheorem{ques}[equation]{Question}
\newenvironment{question}[1][]{\begin{ques}[#1]\pushQED{\qed}}{\popQED \end{ques}}
\usepackage[in]{fullpage}
\usepackage[colorlinks=true, pdfstartview=FitV, linkcolor=black, citecolor=black, urlcolor=black]{hyperref}

\newcommand{\cN}{\mathcal{N}}

\newcommand{\cP}{\mathcal{P}}

\newcommand{\bd}{\mathbf{d}}

\newcommand{\bp}{\mathbf{p}}

\newcommand{\bq}{\mathbf{q}}

\newcommand{\za}{\ensuremath{\alpha}}

\newcommand{\ZZ}{\mathbb{Z}}

\newcommand{\NN}{\mathbb{N}}
\newcommand{\dd}{\textbf{d}}

\renewcommand{\phi}{\varphi}
\renewcommand{\emptyset}{\varnothing}

\renewcommand{\tilde}[1]{\widetilde{#1}}
\newcommand{\ol}[1]{\overline{#1}}

\newcommand{\setst}[2]{\left\{ #1 \mid #2 \right\}}

\newcommand{\comment}[1]{}

\makeatletter
\def\Ddots{\mathinner{\mkern1mu\raise\p@
\vbox{\kern7\p@\hbox{.}}\mkern2mu
\raise4\p@\hbox{.}\mkern2mu\raise7\p@\hbox{.}\mkern1mu}}
\makeatother

\newcommand{\Hom}{\operatorname{Hom}}

\DeclareMathOperator{\rad}{rad}

\DeclareMathOperator{\Ext}{Ext}

\DeclareMathOperator{\End}{End}

\DeclareMathOperator{\Aut}{Aut}

\DeclareMathOperator{\ind}{ind}

\DeclareMathOperator{\rep}{rep}
\DeclareMathOperator{\Mat}{Mat}

\newcommand{\kk}{\Bbbk}
\DeclareMathOperator{\soc}{soc}
\DeclareMathOperator{\id}{id}

\begin{document}
\title{On algebras of finite general representation type}
\author{Ryan Kinser}
\address{University of Iowa, Department of Mathematics, Iowa City, IA, USA}
\email[Ryan Kinser]{ryan-kinser@uiowa.edu}

\author{Danny Lara}
\address{Texas State University, Department of Mathematics, San Marcos, TX, USA}
\email[Danny Lara]{d\_l396@txstate.edu}

\begin{abstract}
We introduce the notion of ``finite general representation type'' for a finite-dimensional algebra, a property related to the ``dense orbit property'' introduced by Chindris-Kinser-Weyman. We use an interplay of geometric, combinatorial, and algebraic methods to produce a family of algebras of wild representation type but finite general representation type.
For completeness, we also give a short proof that the only local algebras of discrete general representation type are already of finite representation type.
We end with a Brauer-Thrall style conjecture for general representations of algebras.
\end{abstract}

\makeatletter
\@namedef{subjclassname@2020}{%
  \textup{2020} Mathematics Subject Classification}
\makeatother



\maketitle

\setcounter{tocdepth}{1}
\tableofcontents

\section{Introduction}
\subsection{Context and motivation}
Throughout the paper, $\kk$ denotes a field which we assume is algebraically closed to facilitate discussion, although it can be arbitrary in our main result.  An \emph{algebra} is always associative, and usually finite-dimensional over $\kk$.

A fundamental result in representation theory of algebras is the \emph{tame/wild dichotomy} \cite{Drozd}.  Further differentiating finite representation type algebras within tame representation type, we can informally summarize the idea by saying that every finite-dimensional algebra $A$ falls into exactly one of the following three classes:
\begin{description}
\item[Finite] $A$ has only finitely many isomorphism classes of indecomposable representations;

\item[Tame] not of finite type, but for each fixed dimension $d$, there exist finitely many families of representations, each depending on at most one parameter from $\kk$, such that every isomorphism class of indecomposable representations of $A$ of dimension $d$ appears in one of these families;

\item[Wild] given an arbitrary $N \in \ZZ_{\geq_0}$, there exists a dimension $d$ such that $A$ has a family of isomorphism classes of indecomposable representations depending on at least $N$ parameters from $\kk$. 
\end{description}

This description is only informal because we have not made precise the meaning of families of representations depending on a certain number of parameters.
We can geometrically formalize it following \cite[\S1.9]{Kac83}.
If $G$ is a connected algebraic group acting on an irreducible algebraic variety $X$, Rosenlicht's theorem says there exists an open subset $X^\circ \subset X$ which has a geometric quotient $X^\circ \to X^\circ // G=:Z$.
More generally if $X$ is a constructible subset of a variety, we can apply Rosenlicht's theorem to each irreducible component $Y_1, \dotsc, Y_s$ of $X$ 
and get geometric quotients $Y_i^\circ \to Z_i$.  We define the \emph{general number of parameters} of $G$ on $X$ as $\mu^\circ(G,X):=\max_i \dim Z_i$ (which does not depend on the choices of $Y_i^\circ$). 
For example, $\mu^\circ(G,X)=0$ if and only if every irreducible component of $X$ has a dense $G$-orbit.
We can iterate this procedure,
replacing $X$ with $X \setminus (Y_1^\circ \cup \dotsc \cup Y_s^\circ)$ after each step, which is again a constructible set, but of strictly smaller dimension than $X$, so the process terminates.
However, the geometric quotients in later steps may actually have larger dimension than those appearing in earlier steps.
We define $\mu(G,X)$ as the maximum dimension of any geometric quotient arising in this process.  We say that the set of orbits of $G$ on $X$ \emph{depends on $\mu(G,X)$ parameters}.

Using Bongartz's geometric Morita equivalence \cite{Bongartz91}, we may as well restrict our attention to algebras which are admissible (in particular, finite dimensional) quotients of quiver path algebras, $A=\kk Q/I$.  We assume $A$ is of this form throughout the paper.
Then the general situation above is applicable to the action of the base change group $GL(\bd)$ on the variety of representations of $A$ of dimension vector $\bd$, denoted $\rep_A(\bd)$.  For each $\bd$, the indecomposable representations form a $GL(\bd)$-invariant constructible subset $\ind_A(\bd) \subset \rep_A(\bd)$.
We write 
\begin{equation}
    \mu_A(\bd) = \mu(GL(\bd), \ind_A(\bd))
    \quad\text{and}\quad
    \mu^\circ_A(\bd) = \mu^\circ(GL(\bd), \ind_A(\bd)).
\end{equation}
By convention we take $\mu_A(\bd)=0$ if $\ind_A(\bd)=\emptyset$.

When $A=\kk Q$ we simply use $Q$ as the subscript instead of $\kk Q$.
In this case, studying behavior of \emph{general representations} of quivers, meaning behavior on dense subsets of each $\rep_Q(\bd)$, has a rich history including the works \cite{Kac83,Schofield92,CB96,CB96b,DW02,DSW07, BHZT09,DF15,GHZ18}.
A number of classical results on representations of quivers can be formulated in terms of numbers of parameters as follows.  A quiver is of:
\begin{enumerate}
    \item finite representation type if and only if $\mu_Q(\bd)=0$ for all $\bd$;
    \item tame representation type if and only if $\mu_Q(\bd)\leq 1$  for all $\bd$;
    \item wild representation type if and only if $\mu_Q(\bd)$ is unbounded as $\bd$ varies.
\end{enumerate}
The remarkable fact that there is no intermediate situation between (2) and (3) is a special case of the tame-wild dichotomy for arbitrary finite-dimensional algebras, first proved by Drozd \cite{Drozd,CB88}.
And though we know of no direct reference, it is well understood that all the statements above remain true when $\mu_Q$ is replaced by $\mu_Q^\circ$, that is, when dealing with general representations instead of all representations of quivers.

But the situation becomes much more interesting and subtle when moving to arbitrary $A$.
Consider (1), for example: the fact that finite representation type implies $\mu_A(\bd)=0$ for all $\bd$ is immediate.  But conversely, assuming $\mu_A(\bd)=0$ for all $\bd$ only gives immediately that there are finitely many indecomposable representations of each dimension vector $\bd$, without giving finitely many overall.  We would need to know that a representation infinite algebra admits a dimension vector with infinitely many indecomposables; this fact is highly nontrivial as it implies (along with \cite{Smalo80}) the second Brauer-Thrall conjecture, a difficult theorem first proved by Bautista \cite{Bautista85}.
On the other hand, statements (2) and (3) still hold essentially by definition when replacing $Q$ with $A$.

This brings us to the topic of the present paper, general representations and the functions $\mu^\circ_A$ for arbitrary $A$.
We are not aware of any major results in the generality of the above mentioned ones.
We contribute in this direction by constructing examples of algebras that display behavior not seen in any of the above situations.
Our algebras $A$ are of wild representation type, but have $\mu^\circ_{A}(\bd) = 0$ for all $\bd$ (called the \emph{dense orbit property} in \cite{CKW15}).
The representation theoretic significance of this is that classification of all finite-dimensional representations of such an algebra up to isomorphism is infeasible, yet becomes a finite problem when restricting to general representations of any given dimension vector.  We call such algebras \emph{discrete general representation type}, and \emph{finite general representation type} when there are only finitely many indecomposable representations whose orbits in $\rep_A(\bd)$ are open subvarieties.

\subsection{Main result and related literature}
There are very few wild algebras in the literature which are known to have discrete or finite general representation type (surveyed below).  This is likely because nice algebraic tools for investigating this property have not yet been developed.  One of the motivations of this paper is to broaden the example base, from which a broader understanding can be pursued.  Our main result is the following theorem about the algebras $\Lambda(m,n)$ defined below,
which are known to be of wild representation type  \cite{HM88} except when $(m,n)\leq (4,4)$ or $(m,n)\leq (6,3)$ or $(m,n)\leq (3,6)$ in the standard order on $\ZZ_{\geq 0}^2$.

\begin{theorem}\label{thm:main}
Consider the algebras $\Lambda(m,n)$ given by the following quiver with relations.
\begin{equation}\label{eq:mainalgebras}
\vcenter{\hbox{
\begin{tikzpicture}[point/.style={shape=circle,fill=black,scale=.5pt,outer sep=3pt},>=latex]
\node[point,label={below:$1$}] (1) at (0,0) {} edge[in=135,out=225,loop] node[left] {$a$} () ;
\node[point,label={below:$2$}] (2) at (2,0) {} edge[in=45,out=-45,loop] node[right] {$b$} ();
\path[->] (1) edge node[above] {$c$} (2) ;
\end{tikzpicture} }}
\qquad 
a^m = b^n = ca - bc = b^2 c = 0
\end{equation}
The algebra $\Lambda(m,n)$ is of finite general representation type.
\end{theorem}

We end the paper in \S\ref{sec:future} with a discussion of how these algebras relate to others appearing in the literature with interesting properties, and propose a Brauer-Thrall style conjecture for general representations (Conjecture \ref{conj:discretefinite}).

Before proving this theorem, we first provide a complete proof that the only local algebras of discrete general representation type are already of finite representation type (Proposition \ref{prop:localDO}).
We do this by showing that for a local algebra given by a quiver with at least 2 loops,
the variety of 2-dimensional representations has an irreducible component without a dense orbit.
The key point which requires some work is to show that the natural family of representations to consider is actually dense in an irreducible component.

We now briefly survey the literature for comparisons.
While the famous Artin-Voigt lemma gives a sufficient homological condition for a representation $M$ to have a dense orbit in its irreducible component (namely, $\Ext_A^1(M,M)=0$), this is far from necessary; just consider 1-dimensional representations of $A=\kk[x]/(x^2)$. 
The property of a representation having a dense orbit in its irreducible component was already being examined in more detail in the article of Dade on ``algebraically rigid modules'' in the proceedings of the second ICRA \cite{Dade79}.

The first examples of wild algebras of discrete general representation type were given in \cite[Thm.~4.1]{CKW15} (there called the \emph{dense orbit property}). 
The authors constructed a family of algebras $\Lambda(n)$, depending on $n \in \ZZ_{>0}$, which are of wild representation type for $n \geq 6$ but have $\mu^\circ_{\Lambda(n)}(\bd) = 0$ for all $\bd$ and $n$.  The proof method revealed that they are even of finite general representation type and gave a classification of general indecomposable representations.

Bobi\'nski recently showed that the ``Birkhoff algebras'' are of finite general representation type \cite[Cor.~3.3]{Bobinski21}, as a part of his proof of a more difficult theorem that all of their representation varieties are irreducible.  
These Birkhoff algebras are given by the same quiver in \eqref{eq:mainalgebras} but specializing to $m=n$ and removing $b^2 c=0$ from the defining relations.  
This means that all of the algebras appearing in Theorem \ref{thm:main} are quotients of Birkhoff algebras, 
providing more evidence for an affirmative answer to \cite[Question~4.5]{CKW15}.
Our setup is also closely related to the ``invariant subspace problem'' studied in  \cite{RS06,RS08,RS08b,KS15,KS17,KKS18,KS18,DMS19}.  The varieties appearing there are open subvarieties of the ones we study.

Finally, we recall a few results indicating where \emph{not} to look for wild algebras of finite (or discrete) general representation type, and mention some related conjectures.  We know of a few classes where it has been shown that algebras can only have discrete general representation type if they are already representation finite: algebras whose Auslander-Reiten quiver has a preprojective component \cite[Thm.~3.5]{CKW15}, special biserial algebras \cite[Prop.~4.4]{CKW15}, and radical square zero algebras \cite[Thm.~7.2]{BCHZ15}.
Furthermore, Mousavand has shown that for a minimal representation-infinite algebra which is biserial or nondistributive, an algebra of discrete representation type must be \emph{brick finite} \cite[Thm.~1.7]{Mousavand} (i.e. has finitely many isomorphism classes of \emph{bricks}, which are representations with endomorphism ring $\kk$). We note that being brick finite is equivalent to being $\tau$-tilting finite \cite{DIJ19}.  
Mousavand also conjectured \cite[Conj.~6.6]{Mousavand} that any algebra of discrete general representation type is brick finite, and
it follows from a theorem of Mousavand and Paquette \cite[Thm.~1.5]{MP} that an algebra of finite general representation type must be brick finite, so Conjecture \ref{conj:discretefinite} of this paper actually implies Mousavand's conjecture above. 
They furthermore conjectured that an algebra being \emph{brick discrete} (i.e. finitely many bricks in each fixed dimension, called \emph{Schur representation finite} in \cite{CKW15}) is equivalent to being brick finite.  

General representations are also intimately connected to moduli spaces of quiver representations; the following list of references is highly incomplete but gives some starting points \cite{King94,Schofield01,Reineke03,CS08,Bob14,HZ14,CC18,CK18,CCKW18}.
We expect that many intricate connections between these finiteness properties (as well as others involving moduli spaces, etc. \cite[\S5]{CKW15}) remain undiscovered, and look forward to future developments in this direction.

\subsection*{Acknowledgements}
We thank Grzegorz Bobi\'nski, Hipolito Treffinger, Kaveh Mousavand, and Charles Paquette for comments.
We also thank Jenna Rajchgot for M2 consultation related to Example \ref{ex:stratacontainment}.
Finally, we appreciate an anonymous referee's diligent reading and thoughtful suggestions which substantially improved the readability and correctness of the final draft.
This work was supported by a grant from the Simons Foundation (636534, RK) and by the National Science Foundation under Award No. DMS-2303334.

\section{Background}
\subsection{Quiver representations and finite-dimensional algebras}
We denote a \emph{quiver} by $Q=(Q_0, Q_1, t, h)$, where $Q_0$ is the \emph{vertex} set, $Q_1$ the \emph{arrow} set, and $t, h\colon Q_1 \to Q_0$ give the \emph{tail} and \emph{head} of an arrow $t\za \xrightarrow{\za} h\za$.
A \emph{representation} $M$ of $Q$ is a collection of (finite-dimensional) $\kk$-vector spaces $(M_z)_{z \in Q_0}$ assigned to the vertices of $Q$, along with a collection of $\kk$-linear maps $(M_\alpha \colon M_{t\alpha} \to M_{h\alpha})_{\alpha \in Q_1}$ assigned to the arrows.  
We recall some key facts here, but for a more detailed recollection we refer the interested reader to standard references such as \cite{ASS06,Schiffler14,DWbook}.

Our convention is that a \emph{path} of positive length in a quiver $Q$ is a sequence of arrows $(\alpha_1, \dotsc, \alpha_r)$ where $t\alpha_i = h\alpha_{i+1}$ for each $i$, so paths are read right to left when following the arrows.  There is also a unique length zero path at each $z \in Q_0$.
Then $Q$ determines a \emph{path algebra} $\kk Q$ as in the references above.
A \emph{relation} $r$ on $Q$ is a $\kk$-linear combination $r=\sum_i c_i p_i$ of paths $p_i$ in $Q$ which all have the same tail and head, and each $c_i\in \kk$.
The category of (left) modules over the algebra $\kk Q/I$ is equivalent to the category of representations of a \emph{quiver with relations} $(Q,R)$,
where $R$ is usually taken to be a minimal set of relations generating $I$.
These equivalences can be used freely without significantly affecting the geometry, as made precise in \cite{Bongartz91}.

Given a nonnegative integer $n$, we write $Q_{\geq n}$ for the set of all paths of $Q$ of length greater than or equal to $n$, and $\kk Q_{\geq n} \subseteq \kk Q$ for the $\kk$-span of this set.
An ideal is \emph{admissible} if $\kk Q_{\geq N} \subset I \subset \kk Q_{\geq 2}$ for some $N \geq 2$.  
Given a finite-dimensional $\kk$-algebra $A$, it is Morita equivalent to a quotient of a path algebra $\kk Q/I$. If $I$ is taken to be admissible (which is always possible), then $Q$ is uniquely determined, 
and the Jacobson radical $\rad(\kk Q/I)$  is spanned by $Q_{\geq 1}$ mod $I$. We always assume that $I$ is admissible when we refer to an algebra of the form $A=\kk Q/I$.

\subsection{Representation varieties}\label{sec:repvariety}
Given a quiver $Q$ and \emph{dimension vector} $\bd \colon Q_0 \to \mathbb{Z}_{\geq 0}$, we study the \emph{representation variety}
\[
\rep_Q(\bd) = \prod_{\za \in Q_1} \Mat_{\bd(h\za),\bd(t\za)}(\kk),
\]
where $\Mat_{m,n}(R)$ denotes the space of matrices with $m$ rows, $n$ columns, and entries in a ring $R$.
We consider the left action of the \emph{base change group}
\[
GL(\bd) = \prod_{z \in Q_0} GL(\bd(z))
\]
on $\rep_Q(\bd)$ given by
\[
g\cdot M = (g_{h\za}M_\alpha g_{t\za}^{-1})_{\za\in Q_1},
\]
where $g = (g_z)_{z \in Q_0} \in GL(\bd)$ and $M = (M_\za)_{\za \in Q_1} \in \rep_Q(\bd)$.

Now consider an algebra $A=\kk Q/I$ with corresponding quiver with relations $(Q,R)$. 
For a relation $r=\sum_i c_i p_i$ on $Q$ and $M \in \rep_Q(\bd)$, we write $M_r=\sum_i c_i M_{p_i}$ where $M_{p_i}$ is the product of the matrices associated to arrows along the path $p_i$ in the appropriate order.
Then the representation variety $\rep_A(\bd)$ is the closed $GL(\bd)$-stable subvariety of $\rep_Q(\bd)$ defined by
\[\rep_A(\bd) = \{ M \in \rep_Q(\bd) \, \mid \, M_r=0, \, \mbox{ for all } r\in R\}.\]
Thus, the points of $\rep_A(\bd)$ are representations of $(Q,R)$ of dimension vector $\bd$.  
Simply from the definitions, $GL(\bd)$-orbits in $\rep_A(\bd)$ are in bijection with isomorphism classes of representations of $A$ of dimension vector $\bd$. 
%For a representation $M$ of $A$ of dimension vector $\bd$, we denote by $O_M$ the orbit in $\rep_A(\bd)$ corresponding to the isomorphism class of $M$, and by $\ol{O}_M$ the closure of this orbit.

The varieties $\rep_A(\bd)$ are not necessarily irreducible.
We say that an irreducible component $C \subseteq \rep_A(\bd)$ is an \emph{indecomposable irreducible component} if $C$ has a dense subset whose points correspond to indecomposable representations of $A$.
The first term defined below is equivalent to the \emph{dense orbit property} of \cite{CKW15}.  
We use a different name and definition to emphasize the representation theoretic significance.

\begin{definition}\label{def:discretegeneral}
An algebra $A$ is of \emph{discrete general representation type} if, for each dimension vector $\mathbf{d}$, there is a dense subset of the variety $\rep_A(\mathbf{d})$ which has finitely many $GL(\mathbf{d})$-orbits.  
In other words, every irreducible component of each variety $\rep_A(\mathbf{d})$ has a dense orbit.

Such an algebra is furthermore of \emph{finite general representation type} if, in addition, it has only finitely many isomorphism classes of indecomposable representations whose orbits are dense their corresponding irreducible component.  In other words, if every irreducible component of each variety $\rep_A(\mathbf{d})$ has a dense orbit, and only finitely many of these (varying over all $\bd$) are indecomposable irreducible components.
\end{definition}

As noted in \cite[\S2.1]{CKW15}, the following is a straightforward consequence of the geometric Krull-Schmidt theorem of Crawley-Boevey and Schr\"oer \cite[Thm.~1.1]{CBS02}.

\begin{corollary}\label{cor:indecompcomponent}
An algebra $A$ is of discrete general representation type if and only if each indecomposable irreducible component of each $\rep_A(\bd)$ has a dense orbit.
\end{corollary}

\subsection{Algebraic geometry}\label{sec:AG}
We include here definitions of basic algebraic geometry terminology which can be reviewed in \cite[Ch.~AG]{Borel}.
To say a statement is true for a \emph{general point} of an algebraic variety $X$ means that there exists a dense, open subset $U\subseteq X$ such that the statement is true for all $x \in U$.  If $\alpha\colon X\to Y$ is a morphism of varieties, the \emph{fiber of $\alpha$ over $y \in Y$} is the set $\alpha^{-1}(y)$, which is a closed subvariety of $X$ \cite[\S10.1]{Borel}. An \emph{irreducible subset} of a topological space is one which cannot be expressed as a union of two proper, nonempty closed subsets.  A subset of a topological space is \emph{locally closed} if it is open in its closure.

For $y \in X$ a variety, $T_y X$ denotes the \emph{Zariski tangent space of $X$ at $y$} \cite[\S16]{Borel}.

\subsection{Local algebras}
As a preliminary observation, we show that there are no local algebras which are representation infinite but have discrete general representation type.
We start by recalling a fact from algebraic geometry which can be found, for example, as a special case of \cite[AG~10.1]{Borel}.
\begin{theorem}\label{thm:fiber}
Let $\alpha \colon X \to Y$ be a surjective morphism of irreducible varieties.  Then     for a general point $y \in Y$, the dimension of each irreducible component of the fiber $\alpha^{-1}(y)$ is $\dim X - \dim Y$.

In particular, if an algebraic group $G$ acts on a variety $X$, then for any $x \in X$ we have
\begin{equation}
    \dim G= \dim (G\cdot x) + \dim G_x
\end{equation}
where $G_x$ is the stabilizer of $x$.
\end{theorem}

The next proposition applies the theorem above to help us identify irreducible components of representation varieties, and whether or not they have a dense orbit, by dimension counting.

\begin{proposition}\label{prop:irrcomp}
Let $G$ be a connected algebraic group and let $X$ be a $G$-variety, both over the field $\kk$.  Suppose that $Y \subset X$ is a locally closed, irreducible subset satisfying both:
\begin{enumerate}[(i)]
    \item for each $x \in X$, the intersection $Y \cap (G\cdot x)$ is a finite set;
    \item for all $y \in Y$, we have
    \begin{equation}\label{eq:tangent space inequality}
        \dim_{\kk} \frac{T_y(X)}{T_y(G\cdot y)} \leq \dim Y.
    \end{equation}
\end{enumerate}
Then $\overline{G\cdot Y}$ is an irreducible component of $X$, of dimension $\dim Y + \dim (G\cdot y)$ where $y \in Y$ is any point whose orbit has maximal dimension.

In particular, if $\dim Y > 0$, then $\overline{G\cdot Y}$ does not have a dense orbit.
\end{proposition}
\begin{proof}
Let $C$ be an irreducible component of $X$ containing $G\cdot Y$.  It is enough to show that $\dim G\cdot Y = \dim C$ since $\overline{G \cdot Y}$ is an irreducible closed subvariety of $X$. Let
\begin{equation}
    \varphi\colon G \times Y \to G\cdot Y \subseteq X, \qquad (g, y) \mapsto g\cdot y
\end{equation}
be the action map.
We first wish to show that $\dim \varphi^{-1}(x)$ for $x \in G\cdot Y$ is equal to $\dim G_x$.  If $(g, y) \in \varphi^{-1}(x)$, then $y \in Y \cap G\cdot x$, which is finite by assumption (i).  So the fiber is a disjoint union of constructible subsets
\begin{equation}
    \varphi^{-1}(x) = \coprod_{y \in Y \cap G\cdot x} \setst{g \in G}{g\cdot y = x} \times \{y\}.
\end{equation}
Each set in the disjoint union has the same dimension as $G_x$: to see this, given $y \in Y \cap G\cdot x$ choose any $g_0 \in G$ such that $g_0 \cdot y = x$, and directly check that $\setst{g \in G}{g\cdot y = x} = G_x g_0$.
Thus $\dim \varphi^{-1}(x) = \dim G_x$.

Let $x \in G\cdot Y$ be a general point, and $y \in Y \cap G\cdot x$. By Theorem \ref{thm:fiber}, we have
\begin{equation}
    \dim (G\cdot Y) = \dim(G \times Y) - \dim G_y.
\end{equation}
The second part of Theorem \ref{thm:fiber} then gives the second equality of:
\begin{equation}
    \dim(G\times Y) - \dim G_y = \dim G + \dim Y - \dim G_y = \dim Y + \dim (G\cdot y).
\end{equation}
Since $G\cdot y$ is smooth, we have $\dim G\cdot y = \dim_\kk T_y(G\cdot y)$.  
Combining this with assumption (ii), we get
\begin{equation}
    \dim Y + \dim (G\cdot y) = \dim Y + \dim_\kk T_y(G\cdot y) \geq \dim_\kk T_y(X) \geq \dim_\kk T_y(C)  \geq \dim C.
\end{equation}
Stringing these together, we find $\dim (G\cdot Y) \geq \dim C$.  But $C$ is an irreducible component containing $G \cdot Y$, so the dimensions must be equal and the proof is completed.

For the ``in particular'' statement, if $\dim Y > 0$ then $\dim (G\cdot y) < \dim (\overline{G\cdot Y})$ for all $y \in Y$, thus no orbit can have dimension equal to that of $\overline{G\cdot Y}$.
\end{proof}

In the situation of $GL(\bd)$ acting on $\rep_A(\bd)$, from the definitions we can see that for $M \in \rep_A(\bd)$ the stabilizer $GL(\bd)_M$ is equal to the set of invertible elements of $\End_A(M)$.  This implies that $\dim GL(\bd)_M = \dim_\kk \End_A(M)$, and so
\begin{equation}
    \dim (GL(\bd)\cdot M) = \dim GL(\bd) - \dim_\kk \End_A(M) = \sum_{z\in Q_0} \bd(z)^2 - \dim_\kk \End_A(M).
\end{equation}
While this allows us to compute the dimension of an orbit from representation theoretic data, it is still difficult to determine if the orbit is dense in an irreducible component because dimensions of irreducible components of $\rep_A(\mathbf{d})$ are not readily computable.
The following classical result is a consequence of the Artin-Voigt Lemma.  It is helpful by giving some information about the codimension of an orbit in terms of representation theoretic data.
It can be found in \cite[II.3.5]{Voigt}, for example.

\begin{lemma}\label{lem:AV}
For $M \in \rep_A(\bd)$, we have
\begin{equation}
    \dim_\kk \Ext^1_A(M,M) \geq \dim_\kk \frac{T_M(\rep_A(\bd))}{T_M(GL(\bd)\cdot M)}.
\end{equation}
\end{lemma}

We can now prove the main result of this subsection. 

\begin{proposition}\label{prop:localDO}
A finite-dimensional local algebra is of discrete general representation type if and only if it is representation finite.
\end{proposition}
\begin{proof}
A representation-finite algebra is of finite, thus also discrete, general representation type for elementary reasons.
For the converse, we will show that any representation-infinite local algebra $A=\kk Q/I$ is not of discrete representation type, by exhibiting an irreducible component of $\rep_A(2)$ which does not have a dense orbit.

We reduce to the case $I=\kk Q_{\geq 2}$ where $Q$ is the quiver with 1 vertex and $n \geq 2$ loops as follows: first, $Q$ must have only one vertex since $A$ is local, and at least 2 loops since $A$ is representation infinite.  
Since every representation of $A$ of dimension 2 has Loewy length at most 2, each such representation is annihilated by $\kk Q_{\geq 2}$. 
So we have $\rep_A(2)=\rep_{\kk Q/\kk Q_{\geq 2}}(2)$ as closed subvarieties of $\rep_Q(2)$, and one variety fails to have a dense orbit if and only if the other does.
(We remark they are not generally equal if we considered them as closed subschemes.)
Thus, we now take $I=\kk Q_{\geq 2}$ below.

For $\lambda=(\lambda_1, \dotsc, \lambda_{n-1})\in \kk^{n-1}$, consider the representation of $A$ given by the matrices:
\begin{equation}
    M_\lambda = \left(
    \begin{bmatrix}0 & 1\\0 & 0\end{bmatrix},\
    \begin{bmatrix}0 & \lambda_1\\0 & 0\end{bmatrix},\cdots, \begin{bmatrix}0 & \lambda_{n-1}\\0 & 0\end{bmatrix}\right).
\end{equation}
Then $Y:=\{M_\lambda\ \mid\ \lambda \in \kk^{n-1}\} \subset \rep_A(2)$ is a closed subvariety of dimension $n-1$.  Furthermore, it can be directly checked that $M_\lambda \not\simeq M_\mu$ for $\lambda \neq \mu$, so $Y$ intersects each $GL(2)$-orbit of $\rep_A(2)$ in just one point. It can also be computed that $\dim_\kk\Ext_A^1(M_\lambda, M_\lambda)=n-1$, so applying Lemma \ref{lem:AV} we get
\begin{equation}
\dim_{\kk} \frac{T_{M_\lambda}(\rep_A(2))}{T_{M_\lambda}(GL(2)\cdot {M_\lambda})} 
\leq 
\dim_\kk\Ext_A^1(M_\lambda, M_\lambda)
=n-1=\dim Y.
\end{equation}
Thus  \eqref{eq:tangent space inequality} holds here and Proposition \ref{prop:irrcomp} implies that $\overline{GL(2)\cdot Y}$ is an irreducible component of $\rep_A(2)$ without a dense orbit, since $n-1\geq 1 >0$.
\end{proof}

%%%%%%%%%%%%%%%%%%%%%%%%%%%%%%%%%%%%%%%%%%%%%%%%%%%%%%
\section{Stratifying the representation varieties}
Throughout this section we fix $(m,n)$ and let $\Lambda:=\Lambda(m,n)$ as in \eqref{eq:mainalgebras}.  We also fix a dimension vector $\bd=(d_1,d_2)$.
A representation of $\Lambda$ is typically denoted by a triple $(A,B,C)$ of matrices of appropriate sizes.

\subsection{Background on partitions and nilpotent matrices}
For integers $d,e \geq 0$, we denote by $\mathcal{P}_e(d)$ the set of partitions of $d$ with parts of size at most $e$.
The elements of $\mathcal{P}_e(d)$ are in bijection with isomorphism classes of $\kk[T]/(T^e)$-modules of dimension $d$, where a partition
\begin{equation}
    \bp=(p_1, p_2, \dotsc, p_l), \qquad 0 \leq p_i \leq e, \quad p_i \geq p_{i+1}
\end{equation}
of $d$ corresponds to $\bigoplus_{i=1}^l \kk[T]/(T^{p_i})$.
In this case we say the module, or corresponding nilpotent matrix, \emph{has type $\bp$}.
We also use exponent notation for partitions when convenient, where an entry of $a^b$ in a partition means that a part of size $a$ appears $b$ times (with $b=0$ meaning that a part of size $a$ does not appear). 
We write $\ell(\bp)$ for the number of nonzero parts of $\bp$, and sometimes use the notation $\min(\bp):=p_{\ell(\bp)}$.

We use the dominance partial order on partitions, defined by
\begin{equation}
    \bp \leq \bq \quad \Longleftrightarrow \quad \sum_{i=1}^k p_i \leq \sum_{i=1}^k q_i \quad \forall k=1, 2, 3, \dotsc
\end{equation}
The geometric meaning of this partial order is as follows.
Writing $\cN_d$ for the closed subvariety of $\Mat_{d \times d}(\kk)$ consisting of nilpotent matrices, and $\cN_d^\bp$ for the locally closed subvariety of nilpotent matrices of type $\bp$,
a theorem of Gerstenhaber and Hesselink \cite{Hesselink76} states that
\begin{equation}
    \ol{\cN_d^\bp} \subseteq \ol{\cN_d^\bq}\quad \text{if and only if}\quad \bp \leq \bq.
\end{equation}

\subsection{From representations to matrices over $\kk[T]$}\label{sec:KT}
To make our notation more compact, we reduce from studying triples of matrices $(A,B,C)$ over $\kk$ to single matrices over the polynomial ring $\kk[T]$.
Fixing a $d_1\times d_1$ matrix $A_0$ and $d_2\times d_2$ matrix $B_0$ such that $A_0^m=B_0^n=0$,
we identify $(A_0, \kk^{d_1})$ with a $\kk\left[T\right] /\left(T^m\right)$-module $X$ and
$(B_0, \kk^{d_2})$ with a $\kk\left[T\right]/\left(T^n\right)$-module $Y$. 
The relation $B_0 C = C A_0$ is equivalent to $C$ being a $\kk[T]$-module homomorphism, which we still denote by $C\colon X\to Y$, 
and the relation $B_0^2 C=0$ is equivalent to $T^2 C=0$ with regards to the natural $\kk[T]$-module structure on $\Hom_{\kk[T]}(X,Y)$ (i.e.,
 $C \in \soc^2 \Hom_{\kk[T]}(X,Y)$).

We fix decompositions of $X$ and $Y$ into indecomposable summands:
\begin{equation}
X \cong \bigoplus_{j=1} ^ l J_{p_j}, \qquad 
Y \cong \bigoplus_{i=1} ^ m J_{q_i}
\end{equation}
where $J_k: =\kk[T]/(T^k)$ considered as a $\kk[T]$-module.
An element $C \in \Hom_{\kk[T]}(X,Y)$ then corresponds to a matrix over $\kk[T]$, where row $i$ is labeled by $J_{q_i}$ and column $j$ labeled by $J_{p_j}$.
We will call such a matrix with its row and column labels a \emph{($\kk[T]$-)labeled matrix}.
The entry in a row labeled by $J_l$ and column labeled by $J_k$ represents a $\kk[T]$-module homomorphism between cyclic modules $J_k \to J_l$, and is thus represented by some $f \in \kk[T]$.
The relation $b^2c = 0$ further requires that $f$ is annihilated by $T^2$, and thus $f$ can be taken of the form below, where $a, b \in \kk$:
\begin{equation}\label{eq:f}
    \begin{cases}
    f= a & \text{if }l=1 \\
    f= aT^{l-1} & \text{if }k=1 \\
    f= a T^{l-1} + b T^{l-2} & \text{if }l, k \geq 2.
    \end{cases}
\end{equation}

We now examine how the action of $\Aut_{\kk[T]}(X) \times \Aut_{\kk[T]}(Y)$ on $\Hom_{\kk[T]}(X,Y)$ translates to row and column operations on a $\kk[T]$-labeled matrix.
The following lemma is straightforward.
\begin{lemma}\label{lem:rowandcolumn}
The following operations correspond to actions of elements of the group $\Aut_{\kk[T]}(X) \times \Aut_{\kk[T]}(Y)$ on $\Hom_{\kk[T]}(X,Y)$, via change of basis:
\begin{enumerate}[(i)]
    \item multiplication of a row labeled $J_i$ by an invertible element of $\kk[T]/(T^i)$, and similarly for columns;
    \item row operations replacing a row labeled $J_j$ with the sum of that row and $f$ times any row labeled $J_i$,
    where $f \in \kk[T]$ if $j \leq i$ and $f \in (T^{j-i})$ if $j > i$;
    \item column operations replacing a column labeled $J_j$ with the sum of that row and $f$ times a column labeled $J_i$,
    where $f \in \kk[T]$ if $j \geq i$ and $f \in (T^{i-j})$ if $j < i$.
\end{enumerate}
Furthermore, every action of an element of $\Aut_{\kk[T]}(X) \times \Aut_{\kk[T]}(Y)$ on $\Hom_{\kk[T]}(X,Y)$ can be realized by a finite sequence of operations of the above form.
\end{lemma}

\subsection{Stratifying the representation varieties}
We now begin to narrow down candidates for the irreducible components of the representation varieties $\rep_\Lambda(\bd)$.
Denote by $\rep_{\Lambda}^{\bp,\bq}(\bd)$ the locally closed subvariety of $\rep_{\Lambda}(\bd)$ consisting of points $(A,B,C)$ such that $(A,B) \in \cN_{d_1}^\bp \times \cN_{d_2}^\bq$, so we have a disjoint union
\begin{equation}\label{eq:stratify}
    \rep_\Lambda(\bd) = \coprod_{\bp,\bq} \rep_{\Lambda}^{\bp, \bq}(\bd).
\end{equation}  
For $(A,B) \in \cN_{d_1}^\bp \times \cN_{d_2}^\bq$, define
\begin{equation}\label{eq:Hdef}
 H(A,B) := \left\{C \in \Mat_{d_2 \times d_1}(\kk) \;|\; (A, B, C) \in \rep_{\Lambda}(\bd) \right \}.
\end{equation}
In the language of \S\ref{sec:KT}, we can identify $H(A,B)$ with a vector space
\begin{equation}\label{eq:Hsoc2}
H(A,B)\simeq \soc^2\Hom_{\kk[T]}(X, Y)=\setst{f\colon X\to Y}{T^2f=0}.
\end{equation}
Thus in algebraic geometry language, we have that $H(A,B)$ is isomorphic to the affine space $\mathbb{A}^d$ over the field $\kk$, where $d=\dim_\kk H(A,B)$.

\begin{lemma}
The variety $\ol{\rep_{\Lambda}^{\bp, \bq}(\bd)}$ is irreducible whenever it is nonempty, and each irreducible component of $\rep_\Lambda(\bd)$ is of this form for some $(\bp,\bq)$.
\end{lemma}
\begin{proof}
In short, the irreducibility is because 
the projection map $\pi\colon \rep_{\Lambda}^{\bp, \bq}(\bd) \to \cN_{d_1}^\bp \times \cN_{d_2}^\bq$ realizes 
$\rep_{\Lambda}^{\bp, \bq}(\bd)$ as the total space of a vector bundle over an irreducible base variety, thus it is irreducible and so is its closure.

In more detail, $\pi$ has fiber over $(A,B)$ isomorphic to $H(A,B)$, by the definition \eqref{eq:Hdef}, which is of constant dimension depending only on $(\bp,\bq)$.
Also $\pi$ can be locally trivialized because $H(A,B)$ is the solution space to a system of linear equations whose coefficients depend linearly on the coordinates of the base space.  So $\pi$ is a vector bundle.  A variety is irreducible if and only if admits a covering by irreducible open subvarieties.  Taking $\{U_i\}_{i \in I}$ an open cover of the base which trivializes $\pi$, the covering 
$\{\pi^{-1}(U_i)\}_{i \in I}$ works since by definition of a trivialization each $\pi^{-1}(U_i)\simeq U_i \times \mathbb{A}^d$ where $d=\dim_\kk H(A,B)$ and $(A,B)$ is an arbitrary point of the base.  Each $U_i$ is irreducible since it is an open subvariety of the irreducible variety $\cN_{d_1}^\bp \times \cN_{d_2}^\bq$, and the product of two irreducible varieties is irreducible.  So each $\pi^{-1}(U_i)$ is irreducible.

Then from the decomposition into locally closed subvarieties \eqref{eq:stratify}, we have that the irreducible, closed subvarieties $\ol{\rep_{\Lambda}^{\bp, \bq}(\bd)}$ cover $\rep_\Lambda(\bd)$.
Thus the irreducible components of $\rep_\Lambda(\bd)$ are among these, giving the last statement of the lemma.
\end{proof}

Connecting with the $\kk[T]$-module language, denote by $Z_{A} \subseteq GL(d_1)$ and $Z_{B} \subseteq GL(d_2)$ the centralizers of $A$ and $B$ with respect to conjugation actions, respectively,
so that a stratum $\rep_{\Lambda}^{\bp,\bq}(\bd)$ has a dense $GL(\dd)$-orbit if and only if $H(A,B)$ has a dense $Z_{A} \times Z_{B}$ orbit. 
In the language of \S\ref{sec:KT}, this means that $\rep_{\Lambda}^{\bp, \bq}(\bd)$ has a dense $GL(\dd)$-orbit if and only if $\Hom_{\kk[T]}(X,Y)$ has a dense $\Aut_{\kk[T]}(X) \times \Aut_{\kk[T]}(Y)$-orbit.

Now we wish to further narrow down the pairs $(\bp,\bq)$ which can give rise to irreducible components.
Define $h(\bp,\bq)= \dim H(A,B)$ for any choice of $(A,B) \in \cN_{d_1}^\bp \times \cN_{d_2}^\bq$.
The following lemma is a variation of \cite[Lem.~3.5]{Bobinski21}, so we omit the proof.
The essential idea is that the total space of a vector bundle over an irreducible variety is irreducible, and the condition on the $h$ function allows us to view dense subsets of these strata as contained within the same vector bundle.
Below, we continue to use dominance order of partitions and the product order on pairs of partitions.

\begin{lemma}\label{lem:pqhcontainment}
If $(\bp',\bq')\leq (\bp,\bq)$ and $h(\bp,\bq)=h(\bp',\bq')$ then $\ol{\rep_{\Lambda}^{\bp', \bq'}(\bd)} \subseteq \ol{\rep_{\Lambda}^{\bp, \bq}(\bd)}$.
 \end{lemma}

Given a partition $\bp$, let $\ol{\bp}$ be the partition obtained by replacing each entry $p_i$ with $\min\{p_i, 2\}$.

\begin{lemma}\label{lem:pqbar}
For $(\bp,\bq)$ as above, $h(\bp,\bq)=h(\ol{\bp},\ol{\bq})$.
\end{lemma}
\begin{proof}
For partions with one part, it follows from the description \eqref{eq:Hsoc2} with $X=J_{p_1}$ and $Y=J_{q_1}$ that
\begin{equation}
    h((p_1), (q_1))=\min\{2,\dim\Hom_{\kk[T]}(J_{p_1},J_{q_1})\}
    =\min\{2,p_1,q_1\}=h(\ol{(p_1)},\ol{(q_1)}).
\end{equation} 
Then since $\Hom_{\kk[T]}(X,Y)$ distributes over direct sums, we have
\begin{equation}
    h(\bp,\bq)=\sum_{i,j}h((p_i), (q_j))
    =\sum_{i,j} h(\ol{(p_i)},\ol{(q_j)})=h(\ol{\bp},\ol{\bq}).
    \qedhere
\end{equation}
\end{proof}

%%%%%%%%%%%%%%%%%%%%%%%%%%%%%%%%%%%%%%%%%%%%%%%%%%%%    
\section{Proof of the main theorem} 
\subsection{Overview}
We now set out to prove Theorem \ref{thm:main} using the following general strategy.
Assume $\ol{\rep_{\Lambda}^{\bp, \bq}(\bd)}$ is an irreducible component of $\rep_\Lambda(\bd)$.
We use Lemma \ref{lem:pqhcontainment} when possible to eliminate $(\bp,\bq)$ which do not give rise to irreducible components.
For the remaining $(\bp,\bq)$, we reframe the problem in terms of matrices over $\kk[T]$ as described in Section \ref{sec:KT}.
There we do direct calculations to show that one of two things occurs for each remaining $\rep_{\Lambda}^{\bp, \bq}(\bd)$: either a general representation is decomposable, in which case no further consideration is necessary by Corollary \ref{cor:indecompcomponent},
or $\rep_{\Lambda}^{\bp, \bq}(\bd)$ has a dense orbit.  We accomplish the latter by showing a general element of $\rep_{\Lambda}^{\bp, \bq}(\bd)$ can be reduced to a unique normal form.
These considerations are enough to show that each indecomposbale irreducible component (and possibly other strata) of each $\rep_\Lambda(\bd)$ has a dense orbit, so $\Lambda(m,n)$ is of discrete general representation type.

To finish proving the theorem, it remains to see that there are actually only finitely many indecomposable irreducible components with dense orbits.
We will use further reductions on which $(\bp,\bq)$ give rise to indecomposable irreducible components to show that there is only a finite list of candidates for these.

\begin{remark}
A similar technique was used in \cite{CKW15} but that proof is simpler because \emph{every stratum} for those algebras has a dense orbit.  So containment of closures of strata and potential irreducible components did not need to be considered.  In the algebras considered here, it is not true that every stratum has a dense orbit, so our proof requires additional tools from algebraic geometry to consider strata closure containment.
\end{remark}

\subsection{Initial reductions on $(\bp,\bq)$}
The following lemma is immediate from the description of the algebras we study in \eqref{eq:mainalgebras} and the definition of representation varieties.
Since we are working with $(m,n)$ arbitrary, this allows us to interchange $\bp$ and $\bq$ whenever convenient.

\begin{lemma}\label{lem:duality}
Vector space duality induces an isomorphism of varieties
\begin{equation}
    \rep_{\Lambda(m,n)}^{\bq, \bp}(d_1,d_2) \simeq \rep_{\Lambda(n,m)}^{\bp, \bq}(d_2,d_1).
\end{equation}
which is equivariant with respect to the automorphism of $GL(\bd)$ sending $(g_1,g_2)\mapsto(g_2^{-t},g_1^{-t})$.
Thus one stratum has a dense orbit if and only if the other does, and furthermore a general point of one stratum corresponds to an indecomposable representation if and only if the other does.
\end{lemma}

When both, or neither, of $\bp, \bq$ have a part of size 1, then the situation is relatively straightforward, so we treat these first.

\begin{lemma}\label{lem:pqboth1}
Suppose $\min(\bp)=\min(\bq)=1$.  Then a general representation in $\rep_{\Lambda}^{\bp, \bq}(\bd)$ has a direct summand of dimension vector $(1,1)$, represented by the labeled matrix
\begin{equation}
\bordermatrix{& J_1 \cr 
 J_1 & 1 \cr 
 }.
\end{equation}
\end{lemma}
\begin{proof}
Taking a $\kk[T]$-labeled matrix $C$ corresponding to a general point of $\rep_{\Lambda}^{\bp, \bq}(\bd)$, we can find a row and column both labeled $J_1$ such that the corresponding entry of $C$ is nonzero (say the bottom row and rightmost column).
Then using row and column operations as in Lemma \ref{lem:rowandcolumn}, we can clear all the entries to the left and above this entry:
\begin{equation}
\bordermatrix{&  &  &  & J_1 \cr 
\vdots &\ddots & \vdots& \vdots & \vdots \cr 
J_r & \cdots &*& * & b T^{r-1} \cr 
\vdots &\ddots &  *&  * & \vdots  \cr 
J_1 & \cdots & a_{2} & a_{1} & 1 \cr 
}
\rightsquigarrow \bordermatrix{&  &  &  & J_1 \cr 
&\ddots & \vdots& \vdots & \vdots \cr 
& \cdots &*& * & 0 \cr 
&\cdots &  *&  * & 0  \cr 
J_1 & \cdots & 0 & 0 & 1 \cr 
}
\end{equation}
and thus see a direct summand of dimension vector $(1,1)$ splits off.
In more detail: as in \eqref{eq:f} all the entries in the bottom row are represented by scalars.  Since the rightmost column is labeled $J_1$, Lemma \ref{lem:rowandcolumn}(ii) says its scalar multiples can be added to each column to its left to clear them all. 
Similarly, the entry in a row labeled $J_r$ of the rightmost column can be represented by $bT^{r-1}$ with $b \in \kk$. 
Then Lemma \ref{lem:rowandcolumn}(ii) says that $-bT^{r-1}$ times the bottom row can be added to this row, and like this we clear all terms above the bottom one in the rightmost column.
\end{proof}

\begin{lemma}\label{lem:pqgreater1}
Suppose $k:=\min(\bp)>1$ and $l:=\min(\bq)>1$.  Then a general representation in $\rep_{\Lambda}^{\bp, \bq}(\bd)$ has a direct summand of dimension vector $(k,l)$, represented by the labeled matrix
\begin{equation}
\bordermatrix{& J_k \cr 
 J_l & T^{l-2} \cr 
 }.
\end{equation}

\end{lemma}
\begin{proof}
Taking a $\kk[T]$-labeled matrix $C$ corresponding to a general point of 
$\rep_{\Lambda}^{\bp, \bq}(\bd)$, we can find a row labeled $J_l$ and column labeled $J_k$ such that the corresponding entry of $C$ is nonzero (say the bottom row and rightmost column). 
Since both $k, l > 1$ and $C$ is a general point, this entry is not annihilated by $T$.
Applying automorphisms to each column as in Lemma \ref{lem:rowandcolumn}(i), we can make every entry in the bottom row $T^{l-2}$, then clear them with column operations as in the previous lemma.
Similarly, applying automorphisms to each row as in Lemma \ref{lem:rowandcolumn}(i),
we can make the rightmost entry in any row, say labeled $J_r$, to be $T^{r-2}$.  We can then clear the rightmost entry in such a row as in the previous lemma.

\begin{equation}
\bordermatrix{&  &  &  & J_k \cr 
&\ddots & \vdots& \vdots & \vdots \cr 
J_{r}&\cdots &  *&  * & a_1 T^{r-2} + b_1 T^{r-1} \cr 
& \ddots &*& * & * \cr 
J_{l} & \cdots & * & a_2 T^{l-2} + b_2 T^{l-1} & T^{l-2}  } 
\rightsquigarrow
\bordermatrix{&  &  &  & J_k \cr 
&\ddots & \vdots& \vdots & \vdots \cr 
J_{r}&\cdots &  *&  * &  0 \cr 
& \ddots &*& * & 0 \cr 
J_{l} & \cdots & 0 & 0 & T^{l-2}  }
\end{equation}
In this case, a direct summand of dimension vector $(k,l)$ splits off.
\end{proof}

These two lemmas reduce us to the case that exactly one of $\min(\bp), \min(\bq)$ is equal to 1, and using Lemma \ref{lem:duality}, we can take $\min(\bp)=1$, $\min(\bq)>1$ without loss of generality from here on.
With this, we have another case of a stratum where a general representation is decomposable. The proof is similar to those of Lemmas \ref{lem:pqboth1} and \ref{lem:pqgreater1}, so we omit it.

\begin{lemma}\label{lem:21decomp}
Let $(\bp,\bq)$ be such that $\min(\bp)=1$, $r:=\min(\bq)>1$, and assume $p_k=2$ for some $k$.
Then a general representation in $\rep_{\Lambda}^{\bp, \bq}(\bd)$ has a direct summand of dimension vector $(2,r)$, represented by the labeled matrix
\begin{equation}
\bordermatrix{& J_2 \cr 
 J_r & T^{r-2} \cr 
 }.
\end{equation}
\end{lemma}

With this, we can finally make a dramatic reduction in the possible $(\bp,\bq)$ giving rise to indecomposable irreducible components.

\begin{proposition}\label{prop:pqreduction}
Suppose $(\bp,\bq)$ is such that $\ol{\rep_{\Lambda}^{\bp, \bq}(\bd)}$ is an indecomposable irreducible component of $\rep_\Lambda(\bd)$, and that exactly one of $\bp$ or $\bq$ has a part of size 1.
Then, interchanging $\bp$ and $\bq$ if necessary, the partitions are of the form
\begin{equation}\label{eq:psimplify}
    \bp=(m^{a_m},k^{a_k},1),\qquad 
    a_k \in \{0,1\},\ a_m \in \ZZ_{\geq 0}
\end{equation}
for some $k$ with $2 < k < m$, and
\begin{equation}\label{eq:qsimplify}
    \bq=(n^{b_n},l^{b_l},2^{b_2}), \qquad b_l \in\{0,1\},\ b_2, b_n \in \ZZ_{\geq 0}
\end{equation}
for some $l$ with $2 < l < n$.
\end{proposition}
\begin{proof}
From the above lemmas and indecomposability hypothesis, we can assume $(\bp,\bq)$ satisfies $\min(\bp)=1$, $\min(\bq)>1$, and $\bp$ has no part of size 2.

First we claim that without loss of generality, we can start by assuming
\begin{equation}\label{eq:pqsimplify2}
    \bp=(m^{a_m},k^{a_k},1^{a_1}),\qquad \bq=(n^{b_n},l^{b_l},2^{b_2}), \qquad a_k, b_l \in\{0,1\}
\end{equation}
for some $k$ with $2 < k < m$, and $l$ with $2 < l < n$.
If $\bp$ were not of the form in \eqref{eq:pqsimplify2}, then it would have two parts of size $k', k''$ with $2 < k' \leq k'' < m$.  Replacing these with parts of size $k'-1, k''+1$ would move up in dominance order, but not change the value $h(\bp,\bq)$, which depends only on the number of parts of size 1 and number of parts size greater than or equal to 2 in $\bp$, by Lemma \ref{lem:pqbar}.
From Lemma \ref{lem:pqhcontainment}, the closed subvariety $\ol{\rep_{\Lambda}^{\bp, \bq}(\bd)}$ 
for the $(\bp,\bq)$ which is larger in dominance order would contain the corresponding subvariety for the smaller $(\bp,\bq)$ in dominance order.
Applying a similar argument to $\bq$, we conclude that to be an irreducible component of $\rep_\Lambda(\bd)$,
it must be that $(\bp,\bq)$ is at least of the form \eqref{eq:pqsimplify2} with $a_1$ arbitrary.

Now if $(\bp,\bq)$ is as in \eqref{eq:pqsimplify2} with $a_1>1$,
we could move up in dominance order combining parts of size 1 in $\bp$ into parts of size 2, with possibly one part of size 1 remaining, without changing the value $h(\bp,\bq)$ since $\min(\bq)>1$.
Again in this case Lemma \ref{lem:pqhcontainment} implies that 
$\ol{\rep_{\Lambda}^{\bp, \bq}(\bd)}$ for the largest $(\bp,\bq)$ as in \eqref{eq:pqsimplify2} will contain the strata corresponding to smaller $(\bp,\bq)$ in dominance order, so that we can take $a_1 \in \{0,1\}$.
The indecomposability hypothesis and Lemma \ref{lem:pqgreater1} then force $a_1=1$, so we 
arrive at the forms \eqref{eq:psimplify} and \eqref{eq:qsimplify}.
\end{proof}

\begin{example}
Let $m=5, n=4$, and consider a pair of partitions
\begin{equation}
    \bp'=(4,3^6,1^5) \qquad \bq'=(4,3^5,2).
\end{equation}
Proposition \ref{prop:pqreduction} says that the associated stratum is not an irreducible component, and following the proof we see more specifically that the stratum for $(\bp',\bq')$ is contained in 
$\ol{\rep_{\Lambda}^{\bp, \bq}(\bd)}$ where
\begin{equation}
    \bp=(5^2,4,2^6,1) \qquad \bq=(4^3,3,2^3).
\end{equation}
The general element of this stratum will not be indecomposable since $\bp$ has parts of size 2 and 1, by Lemma \ref{lem:21decomp}.

In this example, we also note Proposition \ref{prop:pqreduction} does not tell us that $\ol{\rep_{\Lambda}^{\bp, \bq}(\bd)}$ is actually an irreducible component, since it could be contained the closure of a stratum associated to $(\bp'',\bq'')$ for some other pair of the forms in \eqref{eq:psimplify} and \eqref{eq:qsimplify}.  Additional tools would be needed to address this.
\end{example}

\subsection{Normal form in remaining strata}
Finally we tackle the remaining strata by showing a general element of those can be reduced to a unique normal form.  
We emphasize that these normal forms are not for general elements of \emph{all} strata, only the ones we have left after our reductions.  In fact, some other strata may not even have a dense orbit.
We introduce the following shorthand notation.

\begin{notation}
An entry of a $\kk[T]$-labeled matrix written as  $T^*$ in a row labeled $J_i$ and column labeled $J_j$ means $T^{i-2}$ if $i,j>1$, and $T^{i-1}$ if $j=1$.
\end{notation}

That is, $T^*$ represents the general homomorphism $J_j\to J_i$ which is annihilated by $T^2$.  Now fix a dimension vector $\bd = (d_1, d_2)$ and let $(\bp, \bq) \in \cP_m(d_1) \times \cP_n(d_2)$. We define a $\kk[T]$-labeled matrix $M_{\bp,\bq}$ as:
\begin{equation}\label{eq:hpq}
M_{\bp, \bq} : = 
\bordermatrix{& \cdots & J_{p_{\ell(\bp)-2}} & J_{p_{\ell(\bp)-1}} & J_{p_{\ell(\bp)}} \cr 
\quad \vdots &\ddots & \ddots& \ddots & \vdots \cr 
J_{q_{\ell(\bp)-2}} & \ddots &T^{ *}& 0 & 0 \cr 
J_{q_{\ell(\bq)-1}} &\ddots &  T^{ *}&  T^{*} & 0 \cr 
J_{q_{\ell(\bq)}} & \cdots & 0 &T^{*} & T^{*} \cr 
} 
\end{equation}

Below we assume $\ol{\rep_{\Lambda}^{\bp, \bq}(\bd)}$ is an irreducible component of $\rep_\Lambda(\bd)$,
and using the duality in Lemma \ref{lem:duality}, we can assume $\bp$ and $\bq$ are as in Proposition \ref{prop:pqreduction}.
Since we want to label rows and columns of matrices with the parts of these partitions, instead of exponential notation we will write 
\begin{equation}
    \bp=(p_1, p_2, \dotsc, p_k,1) \qquad \bq=(q_1, q_2, \dotsc, q_l), \quad p_k \geq 3, \quad q_l \geq 2.
\end{equation}
We can now present the normal form $\kk[T]$-matrices in the remaining strata.

\begin{proposition}
With the setup above, a general element of $\rep_{\Lambda}^{\bp, \bq}(\bd)$ is equivalent to:
\begin{equation}\label{eq:normalform}
M_{\bp,\bq}=\bordermatrix{& &  &  & J_{p_k} & J_1 \cr 
&\ddots & \ddots & \ddots& \ddots & \vdots \cr 
J_{q_{l-3}}& \ddots & T^{q_{l-3}-2} & 0 & 0 & 0 \cr 
J_{q_{l-2}} & \ddots & T^{q_{l-2}-2} & T^{q_{l-2}-2} & 0 & 0  \cr 
J_{q_{l-1}} &\ddots & 0 & T^{q_{l-1}-2} & T^{q_{l-1}-2} &  0   \cr 
J_{q_{l}} & \cdots & 0 & 0 & T^{q_{l}-2} & T^{q_{l}-1} \cr 
}.
\end{equation}
\end{proposition}
\begin{proof}
A general element of $\rep_{\Lambda}^{\bp, \bq}(\bd)$ is represented by a $\kk[T]$-matrix of the form below with $a\neq 0$:
\begin{equation}\label{eq:step0}
\bordermatrix{&  &  &  J_{p_k} & J_1 \cr 
& \ddots & \ddots& \ddots & \vdots \cr 
J_{q_i} &\ddots & *&  * &  b T^{q_i-1} \cr 
& \ddots &*& * & * \cr 
J_{q_l} & \cdots & * & * & a T^{q_l-1} \cr 
}
\qquad a, b \in \kk.
\end{equation}
Call the lower right entry the first \emph{tread}. 
Exactly as in the first step of the proof of Lemma \ref{lem:pqboth1}, we can use this tread and the operations of Lemma \ref{lem:rowandcolumn}(ii) to clear the rightmost column above it and obtain:

\begin{equation}\label{eq:step1}
\rightsquigarrow
\bordermatrix{& &  &  & J_{p_k} & J_1 \cr 
&\ddots & \ddots & \ddots& \ddots & \vdots \cr 
& \ddots & * & * & * & 0 \cr 
J_{q_{l-1}}&\ddots & * & * & * &  0 \cr 
J_{q_l} & \cdots & * & * & a'T^{q_l-2} + b'T^{q_l-1} & T^{q_l-1} \cr 
}\qquad a', b' \in \kk.
\end{equation}

Call the entry to the left of the first tread the first \emph{riser}. 
According to the setup above, we have at most one $J_1$ column, and so the first riser belongs to a column labeled $J_{p_k}$ with $p_k \geq 3$.
Acting on this column by an invertible operation as in Lemma \ref{lem:rowandcolumn}(i),
we get this entry of the form $T^{q_l -2}$. The entries left of the first riser are all of the form
$a''T^{q_l - 2} + b''T^{q_l - 1}$ with $a'',b'' \in \kk$ as in: 
\begin{equation}
\rightsquigarrow
\bordermatrix{& &  &  & J_{p_k} & J_1 \cr 
&\ddots & \ddots & \ddots& \ddots & \vdots \cr 
& \ddots & * & * & * & 0 \cr 
J_{q_{l-1}}&\ddots & * & * & * & 0 \cr 
J_{q_l} & \cdots & * & a'' T^{q_l - 2} + b'' T^{q_l - 1} & T^{q_l -2} & T^{q_l-1} \cr 
} 
\end{equation}
Now this is the key moment where we need all the earlier combinatorial and geometric reductions:
using column operations with the same reasoning as the first step of Lemma \ref{lem:pqboth1},
we zero out the terms to the left of this riser to obtain the following.
\begin{equation}
\bordermatrix{& &  &  & J_{p_k} & J_1 \cr 
&\ddots & \ddots & \ddots& \ddots & \vdots \cr 
& \ddots & * & * & * & 0 \cr 
J_{q_{l-1}}&\ddots & * & * & * &  0 \cr 
J_{q_l} & \cdots & 0 & 0 &  T^{q_l-2} & T^{q_l-1} \cr 
}
\end{equation}
(If another column had been labeled $J_1$ again, this riser would be of the form $T^{q_l-1}$ and not necessarily able to clear entries to its left.)

We can now proceed inductively, labeling each entry above a riser as the next tread, and each entry to left of a tread as the next riser.
Each row and column from here on is labeled by some $J_k$ with $k \geq 2$.
For a tread in a row labeled $J_k$, multiplying this row by an invertible element of $\kk[T]/(T^k)$ makes the entry $T^{k-2}$, and then row operations upward clear the rest of this column above the tread.
Notice this does not change any entries from previous steps since the tread has all zero entries to its right.
Similarly, a riser is used with leftward column operations to clear its row to the left without changing any entries from previous steps, in lower rows.  We illustrate a couple more steps below for the reader.
\begin{equation}
    \rightsquigarrow
\bordermatrix{& &  &  & J_{p_k} & J_1 \cr 
&\ddots & \ddots & \ddots& \ddots & \vdots& \cr 
& \ddots & * & * & 0 & 0 & \cr 
J_{q_{l-1}}&\ddots & * & * & T^{q_{l-1} -2} &  0 &  \cr 
J_{q_l} & \cdots & 0 & 0 &  T^{q_l -2} & T^{q_l-1} \cr 
}
\end{equation}
\begin{equation}\rightsquigarrow
\bordermatrix{& & J_{p_{k-2}} & J_{p_{k-1}} & J_{p_k} & J_1 \cr 
&\ddots & \ddots & \ddots& \ddots & \vdots& \cr 
J_{q_{l-3}}& \ddots & * & 0 & 0 & 0 & \cr 
J_{q_{l-2}} & \ddots & T^{q_{l-2}-2} & T^{q_{l-2}-2} & 0 & 0 & \cr 
J_{q_{l-1}} &\ddots & 0 & T^{q_{l-1}-2} & T^{q_{l-1}-2} &  0 &  \cr 
J_{q_l} & \cdots & 0 & 0 & T^{q_l-2} & T^{q_l-1} \cr 
}
\end{equation}

Depending on the number of parts of $\bp$ and $\bq$, we see that this process will terminate when reaching either the left or top of the matrix, arriving at a normal form $M_{\bp,\bq}$ with no parameters from $\kk$ as in \eqref{eq:normalform}.
\end{proof}

With this normal form, we get a further reduction on $(\bp,\bq)$ giving rise to an indecomposable irreducible component.

\begin{corollary}\label{cor:lengthindecomp}
Suppose $(\bp,\bq)$ is such that $\ol{\rep_{\Lambda}^{\bp, \bq}(\bd)}$ is an indecomposable irreducible component of $\rep_\Lambda(\bd)$.
Then we have
\begin{equation}\label{eq:bpbqlength}
    \ell(\bp) \in \{\ell(\bq)-1,\ \ell(\bq),\ \ell(\bq)+1\}.
\end{equation}
\end{corollary}
\begin{proof}
The condition in \eqref{eq:bpbqlength} is equivalent to $M_{\bp, \bq}$ having no rows or columns entirely of zeros, a clearly necessary condition for the corresponding representation to be indecomposable.
\end{proof}

\begin{remark}\label{rem:transpose}
Recall that Lemma \ref{lem:duality} allowed us interchange $\bp$ and $\bq$ via the equivariant isomorphism $\rep_{\Lambda(m,n)}^{\bq, \bp}(d_1,d_2) \simeq \rep_{\Lambda(n,m)}^{\bp, \bq}(d_2,d_1)$.
It is not literally true that the dense orbit of the stratum $\rep_{\Lambda(m,n)}^{\bq, \bp}(d_1,d_2)$ can be represented by 
the transpose of the $\kk[T]$-matrix $M_{\bp,\bq}\in \rep_{\Lambda(n,m)}^{\bp, \bq}(d_2,d_1)$,
because the powers of $T$ will change. However, using the shorthand with the $T^*$ notation, the dense orbits of these two strata can be represented by matrices which are transposes of one another.
\end{remark}

All together, we have shown that every stratum $\rep_{\Lambda}^{\bp, \bq}(\bd)$ whose closure may potentially be an irreducible component has a dense orbit, represented by the $\kk[T]$-labeled matrix $M_{\bp,\bq}$.
Thus we have proven that the algebras $\Lambda(m,n)$ in Theorem \ref{thm:main} have discrete general representation type.   
To complete the proof of this theorem, it remains to show that there are only finitely many indecomposable irreducible components with dense orbits.  We complete this in the next subsection.

\subsection{Classification of general indecomposables}
Now that we have a normal form for general representations in each stratum which is potentially an indecomposable irreducible component, we are able to determine which among these general representations are indecomposable.  
Throughout this whole section, we assume that $(\bp,\bq)$ is as in Proposition \ref{prop:pqreduction}, retaining that notation, and $(\bp,\bq)$ satisfies \eqref{eq:bpbqlength}.
We can now substantially narrow the possibilities for indecomposable irreducible components by removing those associated to partitions with repeated parts.

\begin{proposition}
Let  $(\bp, \bq) \in \cP_m(d_1) \times \cP_n(d_2)$ be as in Proposition \ref{prop:pqreduction}.
If either $\bp$ or $\bq$ has any repeated parts, i.e. any of $a_n, b_m, b_2$ is greater than 1,
then $M_{\bp, \bq}$ is decomposable. 
\end{proposition}
\begin{proof}
We use row and column operations as in Lemma \ref{lem:rowandcolumn} to zero out some entries of $M_{\bp,\bq}$. 
Let $r$ be the maximal integer such that $p_{r} = p_{r+1}$.  Note that the assumed form of $\bp$ ensures $\ell(\bp)> r+1$ since $\bp$ must end with a 1 that is not repeated.
Recall that we call the top nonzero entry in a column of $M_{\bp,\bq}$ a ``tread'', and we call the leftmost nonzero entry in a row a ``riser''.
Let $s$ be the index of the row containing the riser of column $r$ (which depends on whether $\ell(\bp)=\ell(\bq)$ or $\ell(\bp)=\ell(\bq)+1$).
For readability, we zoom in on a submatrix of $M_{\bp,\bq}$ where all the interactions between nonzero terms occurs, keeping in mind that the row $s+1$ displayed below may not exist, but it would not affect the calculation.

First, the riser in column $r$ can be used to clear the tread in column $r+1$, since $p_r=p_{r+1}$:
\begin{equation}
\bordermatrix{& J_{p_{r-1}} & J_{p_r}  & J_{p_{r+1}} & J_{p_{r+2}} \cr 
J_{q_{s-2}} &T^* & 0 & 0 & 0 \cr 
J_{q_{s-1}} &  T^* &  T^* & 0 & 0\cr 
J_{q_{s}} & 0 &T^* & T^* & 0\cr 
J_{q_{s+1}} & 0 & 0 &T^* & T^* \cr 
}  \rightsquigarrow
\bordermatrix{& J_{p_{r-1}} & J_{p_r}  & J_{p_{r+1}} & J_{p_{r+2}} \cr 
J_{q_{s-2}} &T^* & 0 & 0 & 0 \cr 
J_{q_{s-1}} &  T^* &  T^* & -T^* & 0\cr 
J_{q_{s}} & 0 &T^* & 0 & 0\cr 
J_{q_{s+1}} & 0 & 0 &T^* & T^* \cr 
}.
\end{equation}
Then the tread in row $s$ can clear the entry above it:
\begin{equation}
\rightsquigarrow
\bordermatrix{& J_{p_{r-1}} & J_{p_r}  & J_{p_{r+1}} & J_{p_{r+2}} \cr 
J_{q_{s-2}} &T^* & 0 & 0 & 0 \cr 
J_{q_{s-1}} &  T^* & 0 & -T^* & 0\cr 
J_{q_{s}} & 0 &T^* & 0 & 0\cr 
J_{q_{s+1}} & 0 & 0 &T^* & T^* \cr 
}
\end{equation}
and we see that $M_{\bp,\bq}$ is equivalent to $M_{\tilde{\bp},\tilde{\bq}}\oplus M_{(p_r),(q_s)}$, where $\tilde{\bp},\tilde{\bq}$ is ad hoc notation for the partitions obtained by removing a part of size $p_r$ from $\bp$ and a part of size $q_s$ from $\bq$.

The case of a repeated part in $\bq$ is essentially the same.  In slightly more detail, if $q_{s}=q_{s+1}$ and the tread of row $s$ is in column $r$ of $M_{\bp,\bq}$, then $M_{\tilde{\bp},\tilde{\bq}}\oplus M_{(p_r),(q_s)}$.  Again the fact that $\bp$ has exactly one part of size 1 is relevant to be able to clear the appropriate entries, as it ensures $M_{\bp,\bq}$ has strictly greater than $r$ columns.
\end{proof}

This significantly constrains the possibilities for $(\bp,\bq)$ such that $\ol{\rep_{\Lambda}^{\bp, \bq}(\bd)}$ is an indecomposable irreducible component of $\rep_\Lambda(\bd)$.  In fact, we are left with only the following possibilities and their transposes (as in Remark \ref{rem:transpose}).

\begin{equation}\label{eq:1by1.1}
\bordermatrix{& J_u \cr 
 J_z & T^* \cr 
 }
\qquad
u \in \{1, \dotsc, m\}, 
\quad
z \in \{1,\dotsc, n\}
\end{equation}
%%%%%%%%%%%%%%%%%%
\begin{equation}\label{eq:1by2.1}
\bordermatrix{& J_t & J_u \cr 
 J_z & T^* & T^* \cr 
 }
\qquad
(t,u) \in \{(k, 1)\}_{k=3}^m,  
\quad
z \in \{2,\dotsc, n\}
\end{equation}
%%%%%%%%%%%%%%%%%%
\begin{equation}\label{eq:2by2.1}
\bordermatrix{& J_t & J_u \cr 
J_y & T^{ *}& 0 \cr 
J_z & T^{ *}&  T^{*} \cr 
}
\qquad 
(t,u) \in \{(k, 1)\}_{k=3}^m,  
\quad
(y,z) \in \{(k, 2)\}_{k=3}^n
\end{equation}
%%%%%%%%%%%%%%%%%%%%%%
\begin{equation}\label{eq:2by3.1}
\bordermatrix{& J_s & J_t & J_u \cr 
J_y & T^{ *}&  T^{*} & 0 \cr 
J_z & 0 &T^{*} & T^{*} \cr 
}
\qquad 
(s,t,u) \in \{(m, k,1)\}_{k=3}^{m-1}, 
\quad
(y,z) \in \{(k, 2)\}_{k=3}^n
\end{equation}
 %%%%%%%%%%%%%%%%%%%
\begin{equation}\label{eq:3by3.1}
\bordermatrix{& J_s & J_t & J_u \cr 
J_x & T^{ *}& 0 & 0 \cr 
J_y & T^{ *}&  T^{*} & 0 \cr 
J_z & 0 &T^{*} & T^{*} \cr 
}\qquad 
(s,t,u) \in \{(m, k,1)\}_{k=3}^{m-1}, 
\quad
(x,y,z) \in \{(n,k,2)\}_{k=3}^{n-1}
\end{equation}

Now with this list, we can show that the dense orbit in each of the remaining strata does correspond to an indecomposable module.

\begin{lemma}
If $M_{\bp, \bq}$ is of one of the forms in
\eqref{eq:1by1.1}--\eqref{eq:3by3.1}, then $M_{\bp, \bq}$ is indecomposable. 
\end{lemma}
\begin{proof}
It is enough to show that any $\Lambda$-module endomorphism of one of these $M_{\bp, \bq}$ is of the form $\omega \id_{M_{\bp, \bq}}+ g$ with $\omega \in \kk$ and $g$ a nilpotent endomorphism of $M_{\bp, \bq}$ (see \cite[Cor.~4.20]{Schiffler14}, for example).
An endomorphism of the representation $M_{\bp,\bq}$ can be represented by a pair of $\kk[T]$-matrices $F,G$ where
\begin{equation*}
F=\bordermatrix{ & & & &\cr
&\lambda_1 + T F_{1,1} & T^{p_1 - p_2} F_{1,2} & T^{p_1 - p_3} F_{1,3}  & \cdots  \cr\cr
 &F_{2,1} & \lambda_2 + T F_{2,2} & T^{p_2 - p_3} F_{2,3}  & \ddots  \cr\cr
& F_{3,1} & F_{3,2} & \lambda_3 + T F_{3,3} & \cdots  \cr 
& \vdots &  & \ddots & \ddots  
}\end{equation*}
\begin{equation*}
G=\bordermatrix{ & & & &\cr
&\mu_1 + T G_{1,1} & T^{q_1 - q_2} G_{1,2} & T^{q_1 - q_3} G_{1,3}  & \cdots  \cr\cr
 &G_{2,1} & \mu_2 + T G_{2,2} & T^{q_2 - q_3} G_{2,3}  & \ddots  \cr\cr
& G_{3,1} & G_{3,2} & \mu_3 + T G_{3,3} & \cdots  \cr 
& \vdots &  & \ddots & \ddots  
}\end{equation*}
with  $\lambda_i, \mu_i \in \Bbbk$ and $F_{i,j}, G_{i,j} \in \Bbbk[T]$. 
We will show that $\lambda_i=\lambda_j$,  $\mu_i=\mu_j$,  and then $\lambda_i = \mu_j$ for all $i,j$ through a case-by-case analysis which necessarily uses a number of the reductions made up to here.  
This will prove the lemma because then subtracting $\lambda_1\id_{M_{\bp, \bq}}$ leaves matrices which are strictly lower triangular modulo $(T)$, thus nilpotent endomorphisms of the corresponding $\kk[T]$-modules and of $M_{\bp,\bq}$ itself. 

The $1 \times 1$ case in \eqref{eq:1by1.1} is trivial.

Consider the $1\times 2$ case in \eqref{eq:1by2.1}:
\begin{equation*}
M_{\bp, \bq} = 
\bordermatrix{& J_t & J_u \cr 
 J_z & T^{z-2} & T^{z-1} \cr 
 }
\qquad
(t,u) \in \{(k, 1)\}_{k=3}^m,  
\quad
z \in \{2,\dotsc, n\}.
\end{equation*}
Then the endomorphism ring of $M_{\bp, \bq}$ is determined by equality of the following two matrices over $\kk[T]$:
\begin{equation}
 M_{\bp, \bq}F = \begin{pmatrix} \lambda_1 T^{z-2} + T^{z-1} F_{1,1} + T^{z-1} F_{2,1}  & T^{z-2+t-u} F_{1,2}+ \lambda_2 T^{z-1} + T^{z} F_{2,2} \end{pmatrix}
\end{equation}
\begin{equation}
  G M_{\bp, \bq} = \begin{pmatrix} \mu_1T^{z-2} + T^{z-1}G_{1,1} & \mu_1 T^{z-1} + T^{z} G_{1,1}  \end{pmatrix}.
\end{equation}
yielding polynomial equalities
\begin{align}
     \lambda_1 T^{z-2} + T^{z-1} F_{1,1} + T^{z-1} F_{2,1} & = \mu_1T^{z-2} + T^{z-1}G_{1,1} \label{eq:1} \\
   T^{z-2+t-u} F_{1,2}+ \lambda_2 T^{z-1} + T^{z} F_{2,2} & = \mu_1 T^{z-1} + T^{z} G_{1,1}. \label{eq:2}
\end{align}
Comparing the degree $z-2$ terms in \eqref{eq:1} gives $\lambda_1 = \mu_1$, and comparing degree $z-1$ terms in \eqref{eq:2} gives that $\lambda_2 = \mu_1$, completing this case. The assumptions are essential since $t - u \geq2$  guarantees that $z-2+t-u \not = z-1$ in \eqref{eq:2}.

Next consider the two by two case. 
\begin{equation}\label{eq:2by2.2}
M_{\bp, \bq} = \bordermatrix{& J_t & J_u \cr 
J_y & T^{ y-2}& 0 \cr 
J_z & T^{ z-2}&  T^{z-1} \cr 
}
\qquad 
(t,u) \in \{(k, 1)\}_{k=3}^m,  
\quad
(y,z) \in \{(k, 2)\}_{k=3}^n
\end{equation}
We then obtain matrices 
\begin{align*}
    M_{\bp, \bq} F & = \begin{pmatrix}T^{y-2} \lambda_1 + T^{y-1}F_{1,1} & T^{y - 2 + t - u} F_{1,2} \\
    T^{z-2} \lambda_1 + T^{z-1}F_{1,1} + T^{z-1} F_{2,1} &T^{z - 2 + t - u} F_{1,2} +  T^{z - 1}\lambda_2   + T^zF_{1,1} \end{pmatrix}\\
    G M_{\bp, \bq}  & = \begin{pmatrix}T^{y - 2} \mu_1 + T^{y-1} G_{1,1} + T^{y-2}G_{1,2}  & T^{y - 1} G_{1,2} \\
    T^{y-2} G_{2,1} + T^{z-2}\mu_2 + T^{z-1} G_{2,2}& T^{z-1} \mu_2 + T^z G_{2,2} \end{pmatrix}.
\end{align*}
Comparing $(1,2)$ entries in the equation $M_{\bp, \bq}F =G M_{\bp, \bq}$ gives
\begin{equation}
    T^{y - 2 + t - u} F_{1,2} = T^{y - 1} G_{1,2},
\end{equation}
and since $t-u \geq 2$ this implies $G_{1,2}$ has a constant term equal to zero when we equate degree $y-1$ coefficients. The equations on the diagonal of the matrix equality $M_{\bp, \bq}F =G M_{\bp, \bq}$ are
\begin{align}
T^{y-2} \lambda_1 + T^{y-1}F_{1,1} & = T^{y - 2} \mu_1 + T^{y-1} G_{1,1} + T^{y-2}G_{1,2} \label{eq:2.3} \\
T^{z - 2 + t - u} F_{1,2} +  T^{z - 1}\lambda_2   + T^zF_{1,1}  & = T^{z-1} \mu_2 + T^z G_{2,2} \label{eq:2.4}.
\end{align}
We can infer that $\mu_2 = \lambda_2$ when we equate degree $z-1$ terms in \eqref{eq:2.4}. Again, the assumptions are essential so that $t-u \geq 2$. In \eqref{eq:2.3},  we already know that $G_{1,2}$ has a zero constant term, so comparing degree $y-2$ terms gives $\lambda_1 = \mu_1$. The last equation in the equality $M_{\bp, \bq}F =G M_{\bp, \bq}$ is
\begin{equation}    
T^{z-2} \lambda_1 + T^{z-1}F_{1,1} + T^{z-1} F_{2,1}  =  T^{y-2} G_{2,1} + T^{z-2}\mu_2 + T^{z-1} G_{2,2},
\end{equation}
which yields that $\lambda_1 = \mu_2$ when we equate degree $z-2$ coefficients, completing this case. Note that the assumption $y < z$ is essential to make this claim. 

Now consider the $2 \times 3$ case
\begin{equation}\label{eq:2by3.2}
M_{\bp, \bq} = \bordermatrix{& J_s & J_t & J_u \cr 
J_y & T^{ y-2}&  T^{y-2} & 0 \cr 
J_z & 0 &T^{z-2} & T^{z-1} \cr 
}
\qquad 
(s,t,u) \in \{(m,k,1)\}_{k=3}^{m-1}, 
\quad
(y,z) \in \{(k, 2)\}_{k=3}^n,
\end{equation}
where the matrix equation for $M_{\bp, \bq} F$ is 
{\small
\begin{equation}
 \begin{pmatrix}
T^{y - 2} \lambda_1 + T^{y-1}F_{1,1} + T^{y-2}F_{2,1} 
& T^{y-2 + s - t}F_{1,2} + T^{y-2} \lambda_2 + T^{y-1} F_{2,2} 
& T^{y - 2 + s - u} F_{1,3} + T ^{y - 2 + t - u} F_{2,3} \\

T^{z-2} F_{2,1} + T^{z-1} F_{3,1} 
& T^{z-2} \lambda_2 + T^{z-1} F_{2,2} + T^{z-1} F_{3,2}
& T^{z-2 +t - u} F_{2,3} + \lambda_{3} T^{z-1}  + T^{z} F_{3,3}
\end{pmatrix}
\end{equation}}
and the matrix for $ G  M_{\bp, \bq}$ is
\begin{equation}
   \begin{pmatrix} 
   T^{ y - 2} \mu_1 + T^{y-1}G_{1, 1}
   & T^{y-2} \mu_1 + T^{y-1} G_{1,1} + T^{y-2}G_{1,2}
   & T^{y-1}G_{1,2}\\
   
   T^{y-2} G_{2,1} 
   & T^{y-2} G_{2,1} + T^{z-2} \mu_2  + T^{z-1} G_{2,2}
   & T^{z-1} \mu_2 + T^{z} G_{2,2}
   \end{pmatrix}.
\end{equation}
Comparing $(2, 1)$ entries of $M_{\bp, \bq}F = G  M_{\bp, \bq}$ gives
\begin{align}
  T^{z-2} F_{2,1} + T^{z-1} F_{3,1}   & =    T^{y-2} G_{2,1}, \label{eq:11}
\end{align}
implying that $F_{2,1}$ has a zero constant term when equating degree $z-2$ coefficients. Next consider the following equations obtained from the equality $  M_{\bp, \bq}F=G  M_{\bp, \bq}$ moving diagonally down and right from the top left entry. 

\begin{align}
T^{y - 2} \lambda_1 + T^{y-1}F_{1,1} + T^{y-2}F_{2,1}  & =    T^{ y - 2} \mu_1 + T^{y-1}G_{1, 1}\label{eq:12}\\
T^{z-2} \lambda_2 + T^{z-1} F_{2,2} + T^{z-1} F_{3,2}   & = T^{y-2} G_{2,1} + T^{z-2} \mu_2  + T^{z-1} G_{2,2}.\label{eq:13} 
\end{align}
Since $F_{2,1}$ was shown to have a zero constant term, equating coefficients of degree $y-2$ in \eqref{eq:12} yields that $\lambda_1 = \mu_1$.
Then \eqref{eq:13} gets us that $\lambda_2 = \mu_2$ when equating degree $z-2$ coefficients. 
Comparing $(1,3)$ entries of $M_{\bp, \bq}F = G  M_{\bp, \bq}$ gives
\begin{align}
T^{y - 2 + s - u} F_{1,3} + T ^{y - 2 + t - u} F_{2,3}  &=   T^{y-1}G_{1,2} .\label{eq:14} 
\end{align}
The assumptions are $s- u, t - u \geq 2$ and thus when equate coefficients of degree $y - 1$ we deduce $G_{1,2}$ must have a zero constant term. The final two equalities are
\begin{align}
  T^{y-2 + s - t}F_{1,2} + T^{y-2} \lambda_2 + T^{y-1} F_{2,2}   & = T^{y-2} \mu_1 + T^{y-1} G_{1,1} + T^{y-2}G_{1,2}\label{eq:15} \\
   T^{z-2 +t - u} F_{2,3} + \lambda_{3} T^{z-1}  + T^{z} F_{3,3} & = T^{z-1} \mu_2 + T^{z} G_{2,2}.\label{eq:16}
\end{align}
Equation \eqref{eq:16} immediately gets us that $\lambda_3 = \mu_2$ when we equate degree $z-1$ terms. Note that this is possible since $t- u \geq 2$. Next, recall that $G_{1,2}$ has zero constant term. Thus by equating degree $y-2$ coefficients in \eqref{eq:15}, we conclude that $\lambda_2 = \mu_1$, completing this case.

Consider the final case where
\begin{equation}\label{eq:3by3.2}
M_{\bp, \bq}= \bordermatrix{& J_s & J_t & J_u \cr 
J_x & T^{ x-2}& 0 & 0 \cr 
J_y & T^{ y-2}&  T^{y-2} & 0 \cr 
J_z & 0 &T^{z-2} & T^{z-1} \cr 
}
\qquad 
(s,t,u) \in \{(m,k,1)\}_{k=3}^{m-1}, 
\quad
(x,y,z) \in \{(n, k,2)\}_{k=3}^{n-1}
\end{equation}

 and $M_{\bp, \bq}F$ is equal to 
{\small
\begin{equation}
\begin{pmatrix}T^{x-2}\lambda_1 + T^{x-1} F_{1,1} 
& T^{x-2 + s - t} F_{1,2} 
&T^{x - 2 + s - u} F_{1,3} \\
T^{y - 2} \lambda_1 + T^{y-1} F_{1,1} + T^{y - 2} F_{2,1} 
& T^{y - 2 + s - t}F_{1,2}+T^{y - 2}\lambda_2 + T^{y - 1} F_{2,2} 
& T^{y-2 + s - u} F_{1, 3} + T^{y-2 + t - u } F_{2, 3} \\
T^{z-2} F_{2,1} + T^{z-1} F_{3,1} 
& T^{z-2} \lambda_2 + T^{z-1} F_{2,2} + T^{z-1} F_{3,2} 
& T^{z-2 + t - u} F_{2,3}+  T^{z-1}\lambda_ 3 + T^{z} F_{3,3}\end{pmatrix}
\end{equation} }
and  $G M_{\bp, \bq} $ is equal to
\begin{equation}
 \begin{pmatrix} T^{x-2} \mu_1 + T^{x-1} G_{1,1} + T^{x-2} G_{1,2} 
& T^{x-2} G_{1,2} +T^{x-2} G_{1,3} 
& T^{x-1} G_{1,3} \\

T^{x-2} G_{2,1} + T^{y-2} \mu_2 + T^{y-1} G_{2,2}
& T^{y-2}\mu_2 + T^{y-1} G_{2,2} +T^{y-2} G_{2,3}
& T^{y-1} G_{2,3}\\

T^{x-2} G_{3,1} + T^{y-2} G_{3,2} 
& T^{y-2}G_{3,2} + T^{z-2} \mu_3 + T^{z-1} G_{3,3} 
& T^{z-1}\mu_ 3 + T^z G_{3,3}\end{pmatrix}. \end{equation}
When we equate terms above the diagonal, we obtain the following equations. 
\begin{align}
T^{x-2 + s - t} F_{1,2}   & = T^{x-2} G_{1,2}+ T^{x-2} G_{1,3} \label{eq:3.1}\\
 T^{x - 2 + s - u} F_{1,3}&= T^{x-1} G_{1,3} \label{eq:3.2}\\
T^{y-2 + s - u} F_{1, 3} + T^{y-2 + t - u } F_{2, 3} & =T^{y-1} G_{2,3} .\label{eq:3.3}
\end{align}

By equating coefficients of degree $ y-1$ in \eqref{eq:3.3} gets us that $G_{2,3}$ must have a zero constant. The hypotheses are key because we are guaranteed that $s - u \geq 2$ and $t - u \geq 2$ on the left side of \eqref{eq:3.3}.  
We the equate $x-1$ degree coefficients \eqref{eq:3.2}, $s-u \geq 2$ guarantees that $G_{1,3}$ has a constant term equal to zero. 
Now equating degree $x-2$ coefficients in \eqref{eq:3.1}, using the assumption that $s - t \geq 2$, and also the fact that $G_{1,3}$ has a zero constant term, we can conclude that $G_{1,2}$ must also have zero constant terms.  In summary, we know that $G_{1,3}$, $G_{1,2}$, and $G_{2,3}$ all have zero constant terms. 

We now show that $\lambda_1 = \mu_1, \;\lambda_2 = \mu_2 , \lambda_3 = \mu_3$ by looking at the equations on the diagonal from the equality $M_{\bp, \bq}F = GM_{\bp, \bq}$.  The equations are 
\begin{align}
  T^{x-2}\lambda_1 + T^{x-1} F_{1,1}  &=   T^{x-2} \mu_1 + T^{x-1} G_{1,1} + T^{x-2} G_{1,2} \label{eq:3.4} \\
    T^{y - 2 + s - t}F_{1,2}+T^{y - 2}\lambda_2 + T^{y - 1} F_{2,2}  &=  T^{y-2}\mu_2 + T^{y-1} G_{2,2} +T^{y-2} G_{2,3} \label{eq:3.5}\\
   T^{z-2 + t - u} F_{2,3}+  T^{z-1}\lambda_ 3 + T^{z} F_{3,3}  &=  T^{z-1}\mu_ 3 + T^z G_{3,3}.\label{eq:3.6}
\end{align}
Equating coefficients of degree $z-1$ in \eqref{eq:3.6} and using the hypothesis that $t - u\geq 2$ immediately gets us $\lambda_3= \mu_3$. Using our prior results from \eqref{eq:3.1}-\eqref{eq:3.3}, $G_{2,3}$, $G_{1,2}$ have zero constants. Then by equating the $y - 2$ coefficients in the equations \eqref{eq:3.5}, using the assumption that $s-t \geq 2$, and the fact that $G_{2,3}$ has a zero constant we deduce $\lambda_2 = \mu_2$. We equate degree $x-2$ coefficients in \eqref{eq:3.4} and use the fact that $G_{1,2}$ has a zero constant term to conclude that $\lambda_1 = \mu_1$.

Our next goal is to show $\lambda_1 = \mu_2$ and $\lambda_2 = \mu_3$. We look at the equations below the main diagonal in the equality $M_{\bp, \bq} F =G M_{\bp, \bq}.$

\begin{align}
T^{y - 2} \lambda_1 + T^{y-1} F_{1,1} + T^{y - 2} F_{2,1}  & = T^{x-2} G_{2,1} + T^{y-2} \mu_2 + T^{y-1} G_{2,2}\label{eq:3.7}\\
T^{z-2} F_{2,1} + T^{z-1} F_{3,1}  &=T^{x-2} G_{3,1} + T^{y-2} G_{3,2} \label{eq:3.8}\\
T^{z-2} \lambda_2 + T^{z-1} F_{2,2} + T^{z-1} F_{3,2}  & = T^{y-2}G_{3,2} + T^{z-2} \mu_3 + T^{z-1} G_{3,3} \label{eq:3.9}
\end{align}
Now \eqref{eq:3.9} immediately gives the equality $\lambda_{2} = \mu_3$ when we equate $z-2$ degree coefficients, and
\eqref{eq:3.8} tells us that $F_{2,1}$ has a zero constant term when we equate $z-2$ coefficients. We make this substitution in \eqref{eq:3.7}. We compare degree $y-2$ coefficients in \eqref{eq:3.7} to conclude $\lambda_1  = \mu_2$, completing the final case.
%In all cases,  $F = \lambda_1 M_{\bp, \bq} + F'$ and $G = \lambda_1 M_{\bp, \bq} + G'$ where $F',G'$ are nilpotent  endomorphisms and therefore $(F,G)\omega \id_{M_{\bp, \bq}}+ g$ with $\omega \in \kk$ and $g$ a nilpotent endomorphism of $M_{\bp, \bq}$
\end{proof}

Finally, the results above combine to prove the following theorem, from which Theorem \ref{thm:main} follows.

\begin{theorem}\label{thm:list}
The finite list of $\kk[T]$-matrices 
in \eqref{eq:1by1.1}--\eqref{eq:3by3.1}, 
%with $(\bp, \bq) \in \cP_m(d_1) \times \cP_n(d_2)$, 
along with their transposes as in Remark \ref{rem:transpose},
contains all those modules $M_{\bp,\bq}$ such that $\ol{\rep_{\Lambda}^{\bp, \bq}(\bd)}$ is an indecomposable irreducible component of $\rep_{\Lambda}(\bd)$.
Thus, $\Lambda$ is of finite general representation type.
\end{theorem}

The referee has noted that this list gives a precise upper bound of $9mn -18(m+n)+42$ indecomposable general representations.
The list in the theorem raises the following natural question.

\begin{question}\label{q:whichstrata}
Which of the strata appearing in Theorem \ref{thm:list} are irreducible components of their corresponding representation varieties?
\end{question}

\begin{example}\label{ex:stratacontainment}
For example, it turns out that a representation of the form $M_{(k,1),(1)}$ where $3 \leq k < m$ as in \eqref{eq:1by2.1} is in the closure of the orbit of $M_{(k+1),(1)}$ as in \eqref{eq:1by1.1}.
The only proof we know for this is to construct by trial and error an explicit morphism
\begin{equation}
\phi\colon \mathbb{A}^1 \to \ol{\rep_{\Lambda}^{(k+1), (1)}((k+1, 1))}
\end{equation}
such that $\phi(t)$ is in the orbit of $M_{(k+1),(1)}$ for $t \neq 0$, and $\phi(0) = M_{(k,1),(1)}$.
It seems possible to answer Question \ref{q:whichstrata} using this approach, but may be quite tedious.
\end{example}

\section{Future directions}\label{sec:future}
A natural generalization would be to investigate algebras given by the quiver with relations below,
where $b^q c$ ($q \in \NN$) replaces the relation $b^2 c$ in the family of algebras studied in this paper.
\begin{equation}
\vcenter{\hbox{
\begin{tikzpicture}[point/.style={shape=circle,fill=black,scale=.5pt,outer sep=3pt},>=latex]
\node[point,label={below:$1$}] (1) at (0,0) {} edge[in=135,out=225,loop] node[left] {$a$} () ;
\node[point,label={below:$2$}] (2) at (2,0) {} edge[in=45,out=-45,loop] node[right] {$b$} ();
\path[->] (1) edge node[above] {$c$} (2) ;
\end{tikzpicture} }}
\qquad 
a^m = b^n = ca - bc = b^q c = 0
\end{equation}
These can be seen as ``interpolating'' between the algebras of Theorem \ref{thm:main} and those in \cite{Bobinski21}.
While quite a lot of our approach easily generalizes, the key difficulty seems to be that the conclusion of Proposition \ref{prop:pqreduction} is significantly more complicated for arbitrary $q$.
The idea to search for examples of wild algebras of finite general representation type of this form was inspired by the Hoshino-Miyachi list \cite{HM88}, which summarizes the representation type trichotomy for quotients of path algebras of quivers with two vertices.  This leads us to the following problem.
\begin{problem}
Which algebras in the Hoshino-Miyachi list \cite{HM88}, and more generally which algebras having quivers with two vertices, are of discrete or finite general representation type?
\end{problem}

Finally, we note that the results of this paper and \cite{CKW15, Bobinski21} support the following Brauer-Thrall style conjecture for general representations.  Unpacking, it says that if every irreducible component of a representation variety for $A$ has a dense orbit (i.e. $A$ has the dense orbit property), then there are only finitely many indecomposable irreducible components.

\begin{conjecture}\label{conj:discretefinite}
If $A$ is a finite-dimensional algebra of discrete general representation type (i.e. with the dense orbit property), then $A$ is of finite general representation type.
\end{conjecture}

\bibliographystyle{alpha}
\bibliography{KL-denseorbit}

\end{document}